\documentclass[nosumlimits,twoside]{amsart}%
\usepackage{amsfonts, amsmath, amssymb}
\usepackage[bookmarksnumbered,plainpages,hypertex]{hyperref}
\usepackage{graphicx}
\usepackage{float}
\usepackage{srcltx}
\usepackage[usenames]{color}
\usepackage{amsmath}
\usepackage{amsfonts}
\usepackage{amssymb}%
\providecommand{\U}[1]{\protect\rule{.1in}{.1in}}

\newtheorem{theorem}{\sc Theorem}[section]
\newtheorem{proposition}[theorem]{\sc Proposition}

\newtheorem{lemma}[theorem]{\sc Lemma}
\newtheorem{corollary}[theorem]{\sc Corollary}
\theoremstyle{definition}
\newtheorem{definition}[theorem]{\sc Definition}

\newtheorem{example}[theorem]{\sc Example}

\theoremstyle{remark}
\newtheorem{remark}[theorem]{\sc Remark}

\def\x1{x_1}
\def\x2{x_2}
\def\a1{a_1}
\def\a2{a_2}

\allowdisplaybreaks

\begin{document}
\title[Cocycle deformations]{Cocycle deformations for liftings of quantum linear spaces}
\dedicatory{Dedicated to Mia Cohen on the occasion of her retirement}\author{Alessandro Ardizzoni}
\address{Department of Mathematics, University of Ferrara, Via Machiavelli 35, Ferrara
I-44121, Italy}
\email{alessandro.ardizzoni@unife.it}
\author{Margaret Beattie}
\address{Department of Mathematics and Computer Science, Mount Allison University,
Sackville, NB E4L 1E6, Canada}
\email{mbeattie@mta.ca}
\author{Claudia Menini}
\address{Department of Mathematics, University of Ferrara, Via Machiavelli 35, Ferrara
I-44121, Italy}
\email{men@dns.unife.it}
\thanks{M. Beattie's research was supported by an NSERC Discovery Grant. This paper
was written while A. Ardizzoni and C. Menini were members of G.N.S.A.G.A. with
partial financial support from M.I.U.R. (PRIN 2007).}

\begin{abstract}
Let $A$ be a Hopf algebra over a field $K$ of characteristic $0$ and
suppose there is a coalgebra projection $\pi$ from $A$ to a sub-Hopf
algebra $H$ that splits the inclusion. If the projection is
$H$-bilinear, then $A$ is isomorphic to a biproduct $R \#_{\xi}H$
where $(R,\xi)$ is called a pre-bialgebra with cocycle in the
category $_{H}^{H}\mathcal{YD}$. The cocycle $\xi$ maps $R \otimes
R$ to $H$. Examples of this situation include the liftings of
pointed Hopf algebras with abelian group of points $\Gamma$ as
classified by Andruskiewitsch and Schneider
\cite{ASannals}. One asks when such
an $A$ can be twisted by a cocycle $\gamma:A\otimes A\rightarrow K$
to obtain a Radford biproduct. By results of Masuoka \cite{Mas2,
Mas1}, and Gr\"{u}nenfelder and Mastnak \cite{Grunenfelder-Mastnak},
this can always be done for the pointed liftings mentioned above.

In a previous paper \cite{ABM}, we showed that a natural candidate
for a twisting cocycle is { $\lambda \circ \xi$} where $\lambda\in
H^{\ast}$ is a total integral for $H$ and $\xi$ is as above. We also
computed the twisting cocycle explicitly for liftings of a quantum
linear plane and found some examples where the twisting cocycle we
computed was different from { $\lambda \circ \xi$}. In this note we
show that in many cases this cocycle is exactly $\lambda\circ\xi$
and give some further examples where this is not the case.  As well
we extend the cocycle computation to quantum linear spaces; there is
no restriction on the dimension.
\end{abstract}
\maketitle
\tableofcontents

\section{Introduction}

In \cite{ABM}, the authors studied twistings of Hopf algebras with a
coalgebra projection, i.e., Hopf algebras $A$ with a sub-Hopf
algebra $H$ such that there is an $H$-bilinear projection $\pi$ from
$A$ to $H$ that splits the inclusion. Such Hopf algebras can be
written as  $A \cong R\#_{\xi}H$ where $(R,\xi)$   is a
pre-bialgebra with cocycle $\xi:R\otimes R\rightarrow H$ in the
sense of \cite[3.64]{A.M.S.}.    It was shown that there is a
bijective correspondence between $H$-bilinear cocycles $\gamma\in
Z_{H}^{2}(A,K)$ and suitably defined left $H$-linear cocycles for
the pre-bialgebra $R$. A main theorem was that for $\gamma\in
Z_{H}^{2}(A,K)$ then $A^{\gamma}=R^{\gamma_{R}}\#_{\xi
_{\gamma_{R}}}H$ where $\gamma_{R}\in Z_{H}^{2}(R,K)$ is the
restriction of the cocycle $\gamma$ to the pre-bialgebra $R$ and
$\xi_{\gamma_{R}}$ is suitably defined. One interesting problem then
is to find a cocycle $\gamma\in Z_{H}^{2}(A,K)$, or equivalently
$\gamma_{R}\in Z_{H}^{2}(R,K)$, such that
$A^{\gamma}$ is a Radford biproduct, i.e., such that $\xi_{\gamma_{R}%
}=\varepsilon_{R\otimes R}$. As an example, the twisting cocycle which twists
the Radford biproduct $\mathcal{B}(V)\#k[\Gamma]$ of the group algebra
$K[\Gamma]$ of a finite abelian group $\Gamma$ and the Nichols algebra
$\mathcal{B}(V)$ for $V$ a quantum plane to a lifting of $\mathcal{B}%
(V)\#k[\Gamma]$ whose parameters are all nonzero scalars was explicitly computed.

\par One problem studied in \cite{ABM} is the following. Given a
cocycle $\gamma$ twisting $R\#_{\xi}H$ to a Radford biproduct, is
$\gamma_{R}^{-1}=\Lambda \circ\xi$ where $\Lambda$ is an   integral
on $H$ invariant under the adjoint action of $H$ on itself? It was
shown in \cite{ABM} that if $(\Lambda\circ\xi)^{-1}$ is a cocycle
for $R$, then it does twist $R\#_{\xi}H$ to a Radford biproduct. The
problem is that in general, it may not be a cocycle.

\par Counterexamples with pointed Hopf algebras of dimension $32$ in
\cite{ABM} show that in general it is not true that
$\lambda\circ\xi=\alpha$ is a cocycle that twists the Radford
biproduct to a lifting. In this note, we construct further
counterexamples where the Hopf algebras under consideration are
liftings of a quantum plane. One involves a Hopf algebra of
dimension $128$ and the other a family of Hopf algebras of dimension
$r^{4}s$ where $r,s$ are any integers greater than $1$, $r,s$ not
necessarily distinct. In Theorem \ref{thm: lambdaxi}, we find
sufficient conditions to have $\lambda\circ\xi$ equal to a cocycle
that twists the Radford biproduct to the lifting of a quantum plane.

\par To obtain this result we must first re-visit the computation of
the twisting cocycle done in \cite[Section 5]{ABM} without the
assumption that the parameters are all nonzero, and in Theorem
\ref{thm: connected}, we generalize this construction from the
twisting of a quantum plane to the twisting of any quantum linear
space. In particular, there is no
restriction on the dimension of the space and our theory includes
the case when $q=-1$. We refer the reader also to \cite[Sections
4,5]{Grunenfelder-Mastnak} where the problem of finding the twisting
cocycle is solved by applying ($q$-)exponential maps to Hochschild
cocycles.

\section{Preliminaries}

Throughout this paper, $\Gamma$
will denote a finite abelian group and  $H$ will denote the group algebra $K[\Gamma]$.
Also $W=\oplus_{i=1}^{\theta}Kx_{i}$ will denote a
quantum linear space with $x_{i}\in W_{g_{i}}^{\chi_{i}}$ where $g_{i}%
\in\Gamma$,
$\chi_{i}\in\widehat{\Gamma}$  and $
R:= \mathcal{B}(W)$ will be the Nichols algebra of $W$.
This means that we have:

\begin{itemize}
\item[ ] (QLS I) \hspace{2.3mm} $\chi_{i}(g_{j})\chi_{j}(g_{i})=1$ for $i\neq j$;

\item[ ] (QLS II) \hspace{2mm} $\chi_{i}(g_{i})$ is a primitive $r_{i}$th root
of unity.
\end{itemize}

The following is proved in \cite{AS1} or \cite{bdgconstructing}.

\begin{proposition}
\label{pr: lift} For $W$ a quantum linear space with
$\chi_{i}(g_{i})$ a primitive $r_{i}$th root of 1, all liftings $A:=
A(a_{i},a_{ij}| 1 \leq i,j\leq \theta)$ of
$\mathcal{B}(W)\#K[\Gamma]$ are Hopf algebras generated by the
grouplikes and by $(1,g_{i})$-primitives $x_{i}$, $1 \leq i \leq
\theta$ where
\begin{align*}
hx_{i}  &  = \chi_{i}(h)x_{i}h;\\
x_{i}^{r_{i}}  &  = a_{i}(1 - g_{i}^{r_{i}} );\\
x_{i}x_{j}  &  = \chi_{j}(g_{i})x_{j}x_{i} + a_{ij}(1 - g_{i}g_{j} )
,i\neq j.
\end{align*}
 The last equality implies $a_{ji} = - \chi_{j}(g_{i})^{-1}a_{ij} = -
\chi_{i}(g_{j})a_{ij}$, for all $i\neq j$.

Furthermore, if $a_{i} \neq0$, then
\begin{equation}
\label{eq: condnainonzero}g_{i}^{r_{i}}\neq1 \mbox{ and } \chi_{i}^{r_{i}} =
\varepsilon;
\end{equation}
and if $a_{ij} \neq0$ then
\begin{equation}
\label{eq: condaijnonzero}g_{i}g_{j} \neq1 \mbox{ and
} \chi_{i}\chi_{j} = \varepsilon.
\end{equation}

\end{proposition}

In the next lemma, we list some well-known facts which we  will refer to later.

   \begin{lemma}
\label{re: chi1chi2varepsilon} For $W = \oplus_{i=1}^{\theta}W^{\chi_{i}%
}_{g_{i}}$ a quantum linear space as above, the following hold.

\begin{itemize}
\item[i)] If $\chi_{i}^{r_{i}}=\varepsilon$ (as occurs when $a_{i}\neq0$) then
$g_{i}^{r_{i}}$ commutes with all $x_{j}$.

\item[ii)] If $\chi_{i}\chi_{j}=\varepsilon$ (as occurs when $a_{ij}\neq0$)
and if we set $q:=\chi_{i}(g_{i})$ then

\begin{itemize}
\item[a)] $\chi_{i}(g_{j})=q$ and $\chi_{j}(g_{j})=\chi_{i}^{-1}(g_{j}%
)=\chi_{j}(g_{i})=q^{-1}$.

\item[b)] $g_{i}g_{j}$ commutes with all $x_{k}$ with $k$ different from
$i,j.$

\item[c)] $x_{i}$  commutes with $g_ig_j$ if and only if $q^2 = 1$.
\end{itemize}
\end{itemize}
\end{lemma}

\begin{proof}
$i)$ We have $g_{i}^{r_{i}}x_{j}=\chi_{j}(g_{i})^{r_{i}}x_{j}g_{i}^{r_{i}}=\chi_{i}^{-r_{i}%
}(g_{j})x_{j}g_{i}^{r_{i}}=x_{j}g_{i}^{r_{i}}$.

$ii)$ Part $a)$ is clear. For part $b)$ we have
\[
g_{i}g_{j}x_{k}=\chi_{k}(g_{i}g_{j})x_{k}g_{i}g_{j}=\chi_{i}(g_{k})^{-1}%
\chi_{j}(g_{k})^{-1}x_{k}g_{i}g_{j}=(\chi_{i}\chi_{j})^{-1}(g_{k})x_{k}%
g_{i}g_{j}=x_{k}g_{i}g_{j}.
\]
Let us prove $c)$. By $a),$ we have $g_{i}g_{j}x_{i}=\chi_{i}(g_{i})\chi
_{i}(g_{j})x_{i}g_{i}g_{j}=q^{2}x_{i}g_{i}g_{j}$.
\end{proof}

In \cite[Section 5]{ABM} a cocycle $\alpha$ which twists the Radford biproduct
$A:=\mathcal{B}(V)\#K[\Gamma]=R\#H$, $V$ a quantum plane, to a lifting of this
graded Hopf algebra was computed. Then $A^{\alpha}=R^{\alpha_{R}%
}\#_{\varepsilon_{\alpha_{R}}}H$ is a lifting of $A$ with the given scalars.
However, there it was assumed that the scalars $a_{1},a_{2},a$ were nonzero
and the equalities in Lemma \ref{re: chi1chi2varepsilon} were used freely.

 Formulas using $q$-binomial coefficients $\binom{n}{i}_{q} $
 are essential to the arguments in \cite{ABM}
and may be found in \cite{Kassel}. For technical reasons, if $n,i$
or $n-i$ is negative, we set $\binom{n}{i}_{q} =0$.

\par If $x$ is $(1,g)$-primitive, $g$ grouplike and $gx=qxg$ then by
the $q$-binomial theorem \cite[IV.2.2]{Kassel}
\begin{equation}
  \Delta(x^{N}) =\sum_{0\leq n\leq N}\binom{N}%
{n}_{q}x^{n}g^{N-n}\otimes x^{N-n}   =\sum_{0\leq n\leq
N}\binom{N}{n}_{q^{-1}}g^{n}x^{N-n}\otimes x^{n}  \label{form:
Delta}
\end{equation}

and thus we have also
\begin{equation}
\Delta^{2}\left(  x^{ N}\right)  =\sum_{0\leq n\leq N}%
\sum_{0\leq m\leq N-n}\binom{N}{n}_{q}\binom{N-n}{m}_{q}x^{
n}g^{N-n}\otimes x^{ m}g^{N-n-m}\otimes x^{  N-n-m }. \label{form:
Deltasquare}
\end{equation}
Throughout we work over a field $K$ of characteristic 0, and all maps are
assumed to be $K$-linear. We will use Sweedler notation for the
comultiplication in a $K$-coalgebra $C$ but with the summation sign omitted,
namely $\Delta(x)=x_{(1)}\otimes x_{(2)}$ for $x\in C$. For $C$ a coalgebra
and $A$ an algebra the convolution multiplication in $\mathrm{Hom}(C,A)$ will
be denoted $\ast$. Composition of functions will be denoted by $\circ$ or by
juxtaposition when the meaning is clear.

We assume familiarity with the general theory of Hopf algebras; good
references are \cite{Sw}, \cite{Mo}. Radford biproducts were first introduced
in \cite{Rad}.

\section{The twisting cocycle for the lifting of a quantum linear space}

  In this  section we review the
computations in \cite[Section 5]{ABM} without assuming that the
lifting scalars are all nonzero (so that the conditions in Lemma
\ref{re: chi1chi2varepsilon} are not assumed) and extend the
computations from a quantum plane to a quantum linear space.

Recall that if $A$ is a bialgebra, a convolution invertible map $\gamma
:A\otimes A\rightarrow K$ is called a unital (or normalized) 2-cocycle for $A
$ when for all $x,y,z\in A$,
\begin{align}
\gamma\left(  y_{\left(  1\right)  }\otimes z_{\left(  1\right)  }\right)
\gamma\left(  x\otimes y_{\left(  2\right)  }z_{\left(  2\right)  }\right)
&  =\gamma\left(  x_{\left(  1\right)  }\otimes y_{\left(  1\right)  }\right)
\gamma\left(  x_{\left(  2\right)  }y_{\left(  2\right)  }\otimes z\right)
,\label{form: cocycle1 Hopf}\\
\gamma(x\otimes1)  &  =\gamma(1\otimes x)=\varepsilon_{A}(x).
\label{form: cocycle3}%
\end{align}

For a bialgebra $A$ with a sub-Hopf algebra $H$, we denote by $Z_{H}
^{2}\left(  A,K\right)  $  the space of $H$-bilinear $2$-cocycles
for $A $, i.e., the set of cocycles $\gamma$ as defined above such
that $\gamma(ha \otimes bh^\prime) =
\varepsilon(h)\varepsilon(h^\prime)\gamma(a \otimes b)$ for $a,b \in
A$, $h, h^\prime \in H$.  By \cite[Lemma 4.5]{ABM} every
$H$-bilinear cocycle $\gamma$ is also $H$-balanced, i.e., for $h\in
H$, $\gamma(x h \otimes y) = \gamma(x \otimes hy)$.

If $\gamma\in Z_{H}^{2}(A,K)$ then $A$ with multiplication
$m_{A^{\gamma}  }=\gamma\ast m_{A}\ast\gamma^{-1}$ twisted by
$\gamma$ is denoted $A^{\gamma}$ and is also a bialgebra
\cite{Doi-braided}.

\begin{lemma}
\label{rm: triple} Let $A$ be a lifting of $\mathcal{B}(W) \#
k[\Gamma]$, possibly the trivial one. Suppose that $\gamma: A
\otimes A \rightarrow K$ is $H$-bilinear and $H$-balanced. For $X =
x_{1}^{n_{1}} \ldots x_{\theta }^{n_{\theta}} \in A$ write
$\chi_{X}:= \chi_{1}^{n_{1}} \ast\ldots\ast
\chi_{\theta}^{n_{\theta}}$ and similarly for $Y,Z$. Then for
$h,g,l$ grouplike, checking the
cocycle condition (\ref{form: cocycle1 Hopf}) for a triple
$gX,hY,lZ$ is equivalent to checking for the triple $X,Y,Z$.
\end{lemma}
\begin{proof}
\begin{align*}
&  \gamma(hY_{(1)} \otimes l Z_{(1)})\gamma( gX \otimes hY_{(2)}l Z_{(2)}) =
\gamma(gX_{(1)} \otimes h Y_{(1)})\gamma( gX_{(2)} h Y_{(2)} \otimes lZ)\\
\Longleftrightarrow &  \hspace{1mm} \chi_{Z}(l)\gamma(Y_{(1)} \otimes
Z_{(1)})\gamma( X h\otimes Y_{(2)} Z_{(2)}) = \chi_{Z}(l) \gamma( X_{(1)}h
\otimes Y_{(1)})\gamma( X_{(2)} h Y_{(2)} \otimes Z)\\
\Longleftrightarrow &  \hspace{1mm} \chi_{X}^{-1}(h)\gamma(Y_{(1)} \otimes
Z_{(1)})\gamma( X \otimes Y_{(2)} Z_{(2)}) = \chi_{X}^{-1}(h) \gamma( X_{(1)}
\otimes Y_{(1)}) \gamma( X_{(2)} Y_{(2)} \otimes Z).
\end{align*}
\end{proof}

Recall the following from \cite{ABM}.

\begin{lemma}
\label{lm: qplbasic lemma}\cite[Corollary 5.2]{ABM} \label{co: qlpcocycle} Let
$\gamma\in Z^{2}_{H}(A ,K)$ where $A = \mathcal{B}(V) \# K[\Gamma]$, $V$ a
quantum plane. Suppose that $\chi_{1}\chi_{2} = \varepsilon$ so that the
equalities in Lemma \ref{re: chi1chi2varepsilon}(ii) hold. If $q^{i+k} \neq
q^{j+l} $, then $\gamma^{\pm1}(x_{1}^{i}x_{2}^{j} \otimes x_{1}^{k}x_{2}%
^{l})=0.$
\end{lemma}

\begin{lemma}
\cite[Corollary 4.4]{ABM}\cite[Lemma 1.4]{Chen} \label{co: beattie} For $A$ a
bialgebra, let $\beta\in Z^{2}(A,K)$ and $\gamma\in Z^{2}(A^{\beta},K).$ Then
$\gamma\ast\beta\in Z^{2}(A,K).$
\end{lemma}

\vspace{1mm}

\begin{proposition}
\label{pr: gamma(ai)}   Let $A=R\#H$
where $R,H,W$ are as defined thoughout.   For
 $a_{i}\in K$, $ 1 \leq i \leq
\theta$, define $H$-bilinear
maps $\gamma_{a_{i}}:=\gamma_{i}$, $i=1,\ldots,\theta$, from
$A\otimes A$ to $K$ as follows: $\gamma _{i}=\varepsilon$ on all
$X\otimes Y$ for $X,Y$ elements of $\mathcal{B}(W)$ of the form
$x_{1}^{N_{1}}\ldots x_{\theta}^{N_{\theta}}$ except that for
$0<m<r_{i}$,
\begin{equation}
\gamma_{i}(x_{i}^{m}\otimes x_{i}^{r_{i}-m})=a_{i}, \label{eq: gammai}%
\end{equation}
and $\gamma_{i}$ is then extended to $A\otimes A$ by $H$-bilinearity. If
$a_{i}\neq0$, assume that $\chi_{i}^{r_{i}}=\varepsilon$.

The maps $\gamma_{i}$ lie in $Z^{2}_{H}(A,K)$ and furthermore these cocycles
pairwise commute.
\end{proposition}

\begin{proof}
If $a_{i} = 0$, then $\gamma_{i} = \varepsilon$ and satisfies the conditions
to be a cocycle. Thus we assume that the $a_{i}$ are nonzero and that
$\chi_{i}^{r_{i}} = \varepsilon$. This condition ensures that $\gamma_{i} $ is
$H$-balanced since for any grouplike $h$, $\gamma_{i}(x_{i}^{m} h \otimes
x_{i}^{r_{i} - m}) = \chi_{i}^{-m}(h) a_{i} = \gamma_{i}(x_{i}^{m} \otimes h
x_{i}^{r_{i} - m})$.

By the definition of $\gamma_{i}$, condition (\ref{form: cocycle3}) holds. By
Lemma \ref{rm: triple} it suffices to check condition
(\ref{form: cocycle1 Hopf}) for the triple $X,Y,Z$ where $X,Y,Z$ have the form
$x_{1}^{N_{1}}\ldots x_{\theta}^{N_{\theta}}$. From the definition of
$\gamma_{i}$ both sides of (\ref{form: cocycle1 Hopf}) are $0$ unless all
exponents $N_{j}$, $i\neq j$ are $0$. The proof that $\gamma_{i}$ is a cocycle
now follows as in the proof of \cite[Proposition 5.3]{ABM}.

It remains to show that the cocycles pairwise commute. We show that
$\gamma_{1} $ and $\gamma_{2}$ commute. The proof that $\gamma_{i}$ and
$\gamma_{j}$ commute for any $i \neq j$ is the same. If either $a_{1}$ or
$a_{2}$ is $0$, there is nothing to prove and so we assume that the scalars
are nonzero and $\chi_{i}^{r_{i}} = \varepsilon$, $i = 1,2$. Apply $\gamma_{1}
\ast\gamma_{2}$ and $\gamma_{2} \ast\gamma_{1}$ to $x_{1}^{i}x_{2}^{j} \otimes
x_{1}^{k}x_{2}^{l} $ since $\gamma_{1} \ast\gamma_{2}$ or $\gamma_{2}
\ast\gamma_{1}$ will be $\varepsilon_{R \otimes R} $ applied to any other
elements of $R \otimes R$. Then using the fact that the $\gamma_{i}$ are
$H$-balanced and $H$-bilinear
\begin{align*}
\gamma_{1} \ast\gamma_{2}(x_{1}^{i}x_{2}^{j} \otimes x_{1}^{k}x_{2}^{l})  &  =
\gamma_{1}(x_{1}^{i}g_{2}^{j} \otimes x_{1}^{k})\gamma_{2}(x_{2}^{j} \otimes
x_{2}^{l})\\
&  = \gamma_{1}(x_{1}^{i} \otimes g_{2}^{j}x_{1}^{k})\gamma_{2}(x_{2}^{j}
\otimes x_{2}^{l}) =(\chi_{1}(g_{2}))^{jk}\delta_{i+k,r_{1}}\delta_{j+l,r_{2}%
}a_{1}a_{2} ,
\end{align*}
while
\begin{align*}
\gamma_{2} \ast\gamma_{1}(x_{1}^{i}x_{2}^{j} \otimes x_{1}^{k}x_{2}^{l})  &  =
\gamma_{2}( x_{2}^{j} \otimes g_{1}^{k}x_{2}^{l})\gamma_{1}(x_{1}^{i} \otimes
x_{1}^{k})\\
&  = (\chi_{2}(g_{1}))^{kl}\delta_{i+k,r_{1}}\delta_{j+l,r_{2}}a_{1}a_{2} =
(\chi_{1}(g_{2}))^{-kl}\delta_{i+k,r_{1}}\delta_{j+l,r_{2}}a_{1}a_{2}%
\end{align*}
since $\chi_{1}(g_{2})\chi_{2}(g_{1}) = 1$. If $j+l =r_{2}$ then $\chi
_{1}(g_{2})^{j+l} = \chi_{1}(g_{2})^{r_{2}} = \chi_{2}^{-r_{2}}(g_{1}) = 1$
and so these expressions are equal and the proof is complete.
\end{proof}

\begin{lemma}
\label{rm: liftings} Let $W,A,R,H$
and $\gamma_i:= \gamma_{a_i}$ be as in Proposition \ref{pr:
gamma(ai)}.

i)  Let
$\beta_{i}:=\gamma_{b_i}:A\otimes A\rightarrow K$ for some scalars
$b_{i}$.  Then $\gamma_{i}\ast\beta_{j}=\beta_{j}\ast\gamma_{i}$
for all $i,j$. Moreover $\gamma_{i}\ast\beta_{i}(x_{i}^{m}\otimes x_{i}%
^{n})=\delta_{r_{i},n+m}(a_{i}+b_{i}).$ As a consequence  $\gamma_{a_{i}}%
^{-1}=\gamma_{(-a_{i})}$.

ii)  Let $B$ be the lifting of $A$ with parameter $a_i$. Then
$A^{\gamma_{i}}\cong B$.

iii) For any map $\gamma:A\otimes A\rightarrow K$, we have that
$\gamma\in Z_{H}^{2}(A,K)$ if and only if $\gamma\in
Z_{H}^{2}(A^{\gamma_{i}},K)$.

iv) We have that $\gamma_{1}\ast\ldots\ast\gamma_{\theta}\in Z_{H}^{2}(A,K)$.

v) For any map $\gamma:A\otimes A\rightarrow K$, we have that $\gamma\in
Z_{H}^{2}(A,K)$ if and only if $\gamma\in Z_{H}^{2}(A^{\gamma_{1}\ast
\ldots\ast\gamma_{\theta}},K)$.

vi)  Let $B$ be the lifting of $A$ with parameters
$a_1,\dots,a_{\theta}$. Then
$A^{\gamma_{1}\ast\ldots\ast\gamma_{\theta}}\cong B$.
\end{lemma}

\begin{proof}
i) By Proposition \ref{pr: gamma(ai)}, $\gamma_{i}\ast\beta_{j}=\beta_{j}%
\ast\gamma_{i}$ for all $i\neq j$. It remains to check that $\gamma_{i}%
\ast\beta_{i}=\beta_{i}\ast\gamma_{i}$. But it is only necessary to check
these maps on an element of the form $x_{i}^{m}\otimes x_{i}^{n}$ and note
that $\gamma_{i}\ast\beta_{i}(x_{i}^{m}\otimes x_{i}^{n})=\delta_{r_{i}%
,n+m}(a_{i}+b_{i})=\beta_{i}\ast\gamma_{i}(x_{i}^{m}\otimes x_{i}^{n})$. In
particular, we get $\gamma_{a_{i}}^{-1}=\gamma_{(-a_{i})}$.

ii) As in \cite{ABM}, the multiplication $m_{i}:A^{\gamma_{i}}\otimes
A^{\gamma_{i}}\rightarrow A^{\gamma_{i}}$ changes the product of $x_{i}^{m}$
and $x_{i}^{r_{i}-m}$ for $0<m<r_{i}$, namely
\[
m_{i}(x_{i}^{m}\otimes x_{i}^{r_{i}-m})=\gamma_{i}(x_{i}^{m}\otimes
x_{i}^{r_{i}-m})+x_{i}^{r_{i}}+g_{i}^{r_{i}}\gamma_{i}^{-1}(x_{i}^{m}\otimes
x_{i}^{r_{i}-m})=a_{i}(1-g_{i}^{r_{i}}).
\]
but all other products of elements of the form $x_{i}^{t}$ or
$x_{j}^{s}$ are unchanged.   Thus $A^{\gamma_{i}}\cong B$.  This
 also follows from Proposition
\ref{pro:gamma}.

iii) If $\gamma_{i}\neq\varepsilon$ then $a_{i}\neq0$ so that, by Lemma
\ref{re: chi1chi2varepsilon} part i), $1-g_{i}^{r_{i}}$ commutes with all
$x_{j}$. Let $\gamma:A\otimes A\rightarrow K$. Then by the $H$-bilinearity
condition on cocycles and the fact that $\varepsilon(1-g_{i}^{r_{i}})=0$,
$\gamma\in Z_{H}^{2}(A,K)$ if and only if $\gamma\in Z_{H}^{2}(A^{\gamma_{i}%
},K)$.

iv) Part iii) together with Lemma \ref{co: beattie} implies that for all
$\gamma\in Z_{H}^{2}(A,K)$ one has $\gamma\ast\gamma_{i}\in Z_{H}^{2}(A,K)$
for all $i$. Then $\gamma_{1}\in Z_{H}^{2}(A,K)$ implies $\gamma_{1}\ast
\gamma_{2}\in Z_{H}^{2}(A,K)$ which implies $\left(  \gamma_{1}\ast\gamma
_{2}\right)  \ast\gamma_{3}\in Z_{H}^{2}(A,K)$. Extending this argument, we
obtain $\gamma_{1}\ast\ldots\ast\gamma_{\theta}\in Z_{H}^{2}(A,K)$.

v) From iii), $\gamma\in Z_{H}^{2}(A,K)\Longleftrightarrow\gamma\in Z_{H}%
^{2}(A^{\gamma_{i}},K)$. But then, denoting $B:=A^{\gamma_{i}}$, again using
iii), $\gamma\in Z_{H}^{2}(B,K)\Longleftrightarrow\gamma\in Z_{H}%
^{2}(B^{\gamma_{j}},K)$. Continuing in this way, we obtain that $\gamma\in
Z_{H}^{2}(A,K)$ if and only if $\gamma\in Z_{H}^{2}(A^{\gamma_{1}\ast
\ldots\ast\gamma_{\theta}},K)$.

vi)  This follows by an argument similar to that in ii),
or also follows   from Proposition \ref{pro:gamma}.
\end{proof}

\vspace{1mm}

\begin{proposition}
\label{pr: gamma a} Let $A=R\#H$
with $R , H$ as above.
  Let $i<j$, let $a_{ij}$ be a scalar and assume that
$\chi_{i}\chi_{j}=\varepsilon$ if $a_{ij}\neq0$. Define the
$H$-bilinear map $\gamma_{ij}:=\gamma_{a_{ij}}$ from $A\otimes A$ to
$K$ as follows: $\gamma_{ij}=\varepsilon$ except that  for $0 \leq m
< r_i$,
\[
\gamma_{ij}(x_{j}^{m}\otimes x_{i}^{m})=(m)!_{q}a_{ij}^{m},
\]
and $\gamma_{ij}$ is then extended to all of $A\otimes A$ by $H$-bilinearity.
Then $\gamma_{ij}\in Z_{H}^{2}(A,K)$.
\end{proposition}

\begin{proof}
If $a_{i j}=0$, then $\gamma_{ij}=\varepsilon$. Let $a_{ij}$ be
nonzero so that $\chi_{i}\chi_{j}=\varepsilon$. It is easy to check
that the condition $\chi_{i}=\chi_{j}^{-1}$ ensures that
$\gamma_{ij}$ is $H$-balanced. By the definition of $\gamma_{ij}$
the cocycle conditions need
only be checked on triples of the form $x_{i}^{n_{1}}x_{j}^{m_{1}}%
,x_{i}^{n_{2}}x_{j}^{m_{2}},x_{i}^{n_{3}}x_{j}^{m_{3}}$. The proof is now the
same as the proof of \cite[Proposition 5.6]{ABM} with $q=\chi_{i}(g_{i})$,
with $i$ replacing $1$ and $j$ replacing $2$.
\end{proof}

Note that $\gamma_{ij}^{-1}(x_{j} \otimes x_{i}) = -a_{ij}$; the argument is
in \cite{ABM}.

By Lemma \ref{rm: liftings}, $\gamma_{ij} \ast\gamma_{1} \ast\ldots\ast
\gamma_{\theta}\in Z^{2}_{H}(A,K)$. However the next example shows that
$\gamma_{1} \ast\ldots\ast\gamma_{\theta}\ast\gamma_{ij}$ may not lie in
$Z^{2}_{H}(A,K)$.

\begin{example}
\label{ex: noncommuting cocycles} Let $V$ be a quantum linear space, i.e.
$\theta= 2$, and let $a :=a_{12} \neq0$ and $\chi_{1} \chi_{2} = \varepsilon$.
The same example used to show that $\gamma_{12}$ and $\gamma_{1}$ do not
commute in \cite[Section 5]{ABM} shows that $\gamma_{1} \ast\gamma_{12} \notin
Z^{2}_{H}(A,K)$. We test $\gamma_{1} \ast\gamma_{12}$ on the triple
$x_{1}^{r-1}, x_{2}, x_{1}^{2}$ where $r = r_{1} = r_{2}$ since $\chi_{1}
\chi_{2} = \varepsilon$. Then the left hand side of (\ref{form: cocycle1 Hopf}%
) is
$$
  \left[ (\gamma_{1} \ast\gamma_{12}) (x_{2} \otimes (q+1)x_{1})\right]\left[(\gamma_{1} \ast
\gamma_{12})(x_{1}^{r-1} \otimes x_{1} )\right] = (q+1)aa_{1}.
$$

But the right hand side is
\[
  \left[(\gamma_{1} \ast\gamma_{12})(g_{1}^{r-1} \otimes g_{2})\right]\left[ (\gamma_{1} \ast
\gamma_{12})(x_{1}^{r-1}x_{2} \otimes x_{1}^{2})\right] = q(q+1)aa_{1}.
\]
\qed

\end{example}

Note that if a lifting of a quantum linear space has scalars $a_{ij} \neq0 $
and $a_{ik} \neq0$ (or $a_{ki} \neq0$) then $r_{i} = r_{j} = r_{k} = 2$. For
example, if $a_{12},a_{13}$ are both nonzero, then we must have that
$\varepsilon= \chi_{1}\chi_{2} = \chi_{1}\chi_{3}$ so that $\chi_{2} =
\chi_{3} = \chi_{1}^{-1}$. Then $1 = \chi_{3}(g_{2})\chi_{2}(g_{3}) = \chi
_{2}(g_{2})\chi_{3}(g_{3}) = q_{2}q_{3}$. By Lemma
\ref{re: chi1chi2varepsilon}(ii), $q_{2} = q_{1}^{-1} = q_{3}$. Thus $1 =
q_{2}^{2}$ so $q_{1} = q_{2} = q_{3} = -1$.\newline

\par We can now describe cocycles that twist $\mathcal{B}(W) \#
K[\Gamma]$ to a lifting of this Radford biproduct.

\par Let $B$ be the lifting of $A:= \mathcal{B}(W) \#
K[\Gamma]$ defined by a set of nonzero scalars $a_i$, $a_{ij}$.
Consider $1, \ldots, \theta$ as vertices and construct a nondirected
graph by drawing an edge between $i$ and $j$ if $a_{ij} \neq0$. Let
$C_{i}$ be the set of vertices connected to $i$ by some path. Then
$\{1, \ldots, \theta\} $ is the disjoint union of the connected
components $C_{\tau }$, $\tau\in T$.

Let $\gamma_{\tau}$ be a cocycle such that $\gamma_{\tau}$ is
$\varepsilon$ on $X=x_{1}^{n_{1}}\ldots
x_{\theta}^{n_{\theta}}\otimes x_{1}^{N_{1}}\ldots
x_{\theta}^{N_{\theta}}$ unless $n_{i}=N_{i}=0$ for $i\notin
C_{\tau}$. Then we say that $\gamma_{\tau}$ is a\emph{ cocycle for
the connected component} $C_{\tau}$. The next proposition, together
with Proposition \ref{pr: gamma(ai)}, shows that cocycles of the
form $\gamma _{i},\gamma_{ij}$  belonging to different connected
components commute. Also if $|C_{\tau}|>2$ so that all $r_{i}$ are
$2$, then the cocycles of the types $\gamma_{i},\gamma_{ij}$ for
$C_{\tau}$ also pairwise commute. We noted in Example \ref{ex:
noncommuting cocycles} that this is not true in general.

\begin{proposition}
\label{pr: cocycles commute} For
W,R,H as throughout the paper, let $ A = R \#
H $. For some $i<j$, let $a_{ij}$
be a nonzero scalar and suppose that $\chi_{i}\chi_{j} =
\varepsilon$ and $g_{i}g_{j} \neq1 $.   Assume $r_{i} = r_{j}$.

i) For an integer $k$ let $a_{k}\neq 0$   and suppose that
$\chi_{k}^{r_{k}}= \varepsilon$. Assume that either $k \notin\{i,j
\}$ or $r_{i}=r_{j}=2$. Then
\[
\gamma_{ij} \ast\gamma_{k} = \gamma_{k} \ast\gamma_{ij}.
\]

ii) For integers $k < m$ let $a_{km} \neq 0$   and suppose that
$\chi_{k}\chi_{m} = \varepsilon$ and $r_{k} = r_{m}$. Assume that
either $k,m \notin\{i,j \}$ or $r_{i}=r_{j}=2$. Then
\[
\gamma_{ij} \ast\gamma_{km} = \gamma_{km} \ast\gamma_{ij}.
\]


\end{proposition}

\begin{proof}
  Let $r:= r_{i} = r_{j}$ and let $q:=
\chi_{i}(g_{i})= \chi_{j}(g_{j})^{-1}$.

\par i) Clearly
$\gamma_{ij}\ast\gamma_{k}(X)=\gamma_{k}\ast\gamma_{ij}(X)=0$ unless
the exponent of any $x_{m}$ in $X\in R\otimes R$, $m\neq i,j,k$ is
$0$. Suppose that $i<j<k$ and we test equality of
$\gamma_{ij}\ast\gamma_{k}$ and
$\gamma_{k}\ast\gamma_{ij}$ on $X=x_{i}^{n}x_{j}^{l}x_{k}^{p}\otimes x_{i}%
^{N}x_{j}^{L}x_{k}^{P}$.
\begin{displaymath}
\gamma_{ij}\ast\gamma_{k}(X)   =\gamma_{ij}(x_{i}^{n}x_{j}^{l}g_{k}%
^{p}\otimes x_{i}^{N}x_{j}^{L}g_{k}^{P})\gamma_{k}(x_{k}^{p}\otimes x_{k}%
^{P})
=\chi_{j}^{-lp}(g_{k})\delta_{n,0}\delta_{L,0}\delta_{l,N}\delta
_{p+P,r_{k}}a_{k}(l)!_{q}a^{l}.
\end{displaymath}
and
\begin{displaymath}
\gamma_{k}\ast\gamma_{ij}(X)
=\gamma_{k}(g_{i}^{n}g_{j}^{l}x_{k}^{p}\otimes
g_{i}^{N}g_{j}^{L}x_{k}^{P})\gamma_{ij}(x_{i}^{n}x_{j}^{l}\otimes x_{i}%
^{N}x_{j}^{L})
  =\chi_{k}(g_{i})^{NP}\delta_{n,0}\delta_{L,0}\delta_{l,N}\delta_{p+P,r_{k}%
}a_{k}(l)!_{q}a^{l},
\end{displaymath}
and it remains to show that
$\chi_{j}^{-lp}(g_{k})=\chi_{k}(g_{i})^{NP}$. Note that    if
$\delta_{l,N}  \neq 0$ then
\begin{align*}
\chi_{j}^{-lp}(g_{k}) &  =\chi_{j}^{-Np}(g_{k}) =
 \chi_{i}^{Np}(g_{k})\mbox{  since  }\chi_{i}\chi_{j}=\varepsilon\\
&  =\chi_{k}^{-Np}(g_{i})\mbox{   since   }\chi_{i}(g_{k})\chi_{k}(g_{i})=1.
\end{align*}

But $\chi_{k}(g_{i})^{N(P+p)} = 1 $ if $p+P = r_{k}$ since
$\chi_{k}^{r_{k}} = \varepsilon$. The cases $k<i<j$ and $i<k<j$ are
similar.

\par  Now suppose that $k=i$ so that $r=2$ and $q = -1$. As above, it suffices to  test equality of $\gamma_{ij}\ast\gamma_{i}$ and
$\gamma_{i}\ast \gamma_{ij}$ on $X=x_{i}^{n}x_{j}^{l}\otimes
x_{i}^{N}x_{j}^{L}$. In fact,
since it is clear that both $\gamma_{ij}\ast\gamma_{i}$ and $\gamma_{i}%
\ast\gamma_{ij}$ are $0$ on $X$ unless $L=0$, we assume that $X=x_{i}^{n}%
x_{j}^{l}\otimes x_{i}^{N}$ and, since $N<r =2$, without loss of
generality we may assume that $N=1$. Then
\begin{align*}
\gamma_{ij}\ast\gamma_{i}(X) & =\sum_{k=0}^{1}\binom{1}{k}_{q}\gamma
_{ij}(g_{i}^{n}x_{j}^{l}\otimes x_{i}^{k}g_{i}^{1-k})\gamma_{i}(x_{i}%
^{n}\otimes x_{i}^{1-k})\\
&  =\gamma_{ij}(x_{j}^{l}\otimes g_{i})\gamma_{i}(x_{i}^{n}\otimes
x_{i})+\gamma_{ij}(x_{j}^{l}\otimes x_{i})\gamma_{i}(x_{i}^{n}\otimes1)\\
&  =\delta_{l,0}\delta_{n,1}a_{i}+\delta_{l,1}\delta_{n,0}a_{ij},
\end{align*}
and a similar computation shows that
$\gamma_{i}\ast\gamma_{ij}(X)=\delta
_{l,0}\delta_{n,1}a_{i}+\delta_{l,1}\delta_{n,0}a_{ij}$. Similarly
$\gamma_{ij} \ast \gamma_j = \gamma_j \ast \gamma_{ij}$.\newline

ii)   Set $s:=r_{k} = r_{m}$ and let $q^{\prime}= \chi_{k}(g_{k})$.
First suppose that $i<j<k<m$. It suffices to check equality of these
products of cocycles on an element $X$ of the form $x_{i}^{n}x_{j}^{l}%
x_{k}^{p}x_{m}^{w} \otimes x_{i}^{N}x_{j}^{L}x_{k}^{P}x_{m}^{W}$.
\begin{align*}
\gamma_{ij} \ast\gamma_{km} (X)  &  = \gamma_{ij}(x_{i}^{n}x_{j}^{l}g_{k}%
^{p}g_{m}^{w} \otimes x_{i}^{N} x_{j}^{L} g_{k}^{P}g_{m}^{W}) \gamma_{b}(
x_{k}^{p}x_{m}^{w} \otimes x_{k}^{P} x_{m}^{W})\\
&  = \chi_{j}^{-lw} ( g_{m} ) \delta_{n+L+p+W,0} \delta_{l,N} (l)!_{q}%
a_{ij}^{l} \delta_{w,P}(w)!{q^{\prime}}a_{km}^{w},
\end{align*}
and%

\begin{align*}
\gamma_{km} \ast\gamma_{ij} (X)  &  = \gamma_{km}(g_{i}^{n}g_{j}^{l}x_{k}%
^{p}x_{m}^{w} \otimes g_{i}^{N} g_{j}^{L} x_{k}^{P}x_{m}^{W}) \gamma_{ij}(
x_{i}^{n}x_{j}^{l} \otimes x_{i}^{N} x_{j}^{L})\\
&  = \chi_{k}^{PN}(g_{i}) \delta_{n+L+p+W,0} \delta_{w,P}(w)!{q^{\prime}%
}a_{km}^{w} \delta_{l,N} (l)!_{q}a_{ij}^{l}.
\end{align*}

But, computing as in part i), $\chi_{j}^{-lw}(g_{m}) =
\chi_{i}^{lw}(g_{m}) = \chi_{m}^{-lw}(g_{i}) = \chi_{k}^{lw}(g_{i})$
and since if these expressions are nonzero then $w=P$ and $l=N$, we
are done. The  cases $i < k < j < m$ and $i < k < m < j$ are handled
similarly.

\par Suppose that $k=i$ so that $r=s=2$, and suppose that $i < j <m$. We test equality of $\gamma_{ij}
\ast\gamma_{im}$ and $\gamma_{im} \ast\gamma_{ij}$ on $X =
x_{i}^{n}x_{j}^{l}x_{m}^{p} \otimes x_{i}^{N}x_{j}^{L}x_{m}^{P}$ and
it is clear that both maps applied to $X$ give $0$ unless $n=L=P =
0$. Thus we let $X = x_{j}^{l}x_{m}^{p} \otimes x_{i}^{N} $ and then
we may assume that $N = 1$. Then
\begin{align*}
\gamma_{ij} \ast\gamma_{im} (X)  &  = \sum_{w=0}^{1}
\binom{1}{w}_{q} \gamma_{ij}( x_{j}^{l}g_{m}^{p} \otimes x_{i}^{w}
g_{i}^{1-w}) \gamma_{im}(
x_{m}^{p} \otimes x_{i}^{1-w} )\\
&  = \gamma_{ij}(x_{j}^{l}g_{m}^{p} \otimes
g_{i})\gamma_{ij}(x_{m}^{p} \otimes x_{i}) +
\gamma_{ij}(x_{j}^{l}g_{m}^{p} \otimes x_{i})\gamma
_{im}(x_{m}^{p} \otimes1)\\
&  = \delta_{l,0}\delta_{p,1}a_{im} +
\delta_{l,1}\delta_{p,0}a_{ij},
\end{align*}
and
\begin{align*}
\gamma_{im} \ast\gamma_{ij} (X)  &  = \sum_{w=0}^{1}
\binom{1}{w}_{q} \gamma_{im}( g_{j}^{l} x_{m}^{p} \otimes x_{i}^{w}
g_{i}^{1-w} ) \gamma_{ij}(
x_{j}^{l} \otimes x_{i}^{1-w} )\\
&  =\gamma_{im}( g_{j}^{l} x_{m}^{p} \otimes g_{i})\gamma_{ij}(
x_{j}^{l} \otimes x_{i} ) + \gamma_{im}( g_{j}^{l} x_{m}^{p} \otimes
x_{i} ) \gamma
_{ij}( x_{j}^{l} \otimes1 )\\
&  = \delta_{p,0} \delta_{l,1}a_{ij} +
\delta_{p,1}\delta_{l,0}a_{im}.
\end{align*}
The remaining cases $i<m<j$, etc., have similar proofs.
\end{proof}

Now suppose $B$ is a lifting of $A :=  \mathcal{B}(W)\#K[\Gamma]$
with nonzero scalars $a_i, a_{ij}$. Then by Proposition \ref{pr:
lift}, we must have that if $a_i \neq 0$, $g_i^{r_i} \neq 1$ as well
as $\chi_i^{r_i} = \varepsilon$ and if $a_{ij} \neq 0$ then $g_ig_j
\neq 1$ as well as $\chi_i\chi_j = \varepsilon$. We construct a
cocycle that will twist $A$ to $B$ from the cocycles for the
connected components.

\begin{theorem}
\label{thm: connected}For each connected component $C_{\tau}$ define
$\alpha_{\tau}\in Z_{H}^{2}(A,K)$ by the following:

\begin{itemize}
\item If $|C_{\tau}| = 1$, i.e., $C_{\tau}= \{ i\}$ for some $i$, then define
$\alpha_{\tau}= \gamma_{i}$;

\item If $|C_{\tau}| = 2$, i.e., $C_{\tau}= \{ i,j \}$ for some $i<j$, then
define $\alpha_{\tau}= \gamma_{ij}\ast\gamma_{i}\ast\gamma_{j}$.

\item If $|C_{\tau}| > 2$, i.e., $C_{\tau}= \{ i_{1},i_{2}, \ldots, i_{n} \}$
with $i_{n}<i_{n+1}$, define $\alpha_{\tau}=( \prod_{i<j} \gamma_{ij} )
\ast(\prod_{k=1}^{n} \gamma_{i_{k}})$.
\end{itemize}

Let $\alpha= \prod_{\tau\in T}\alpha_{\tau}$. Then $\alpha$ is a
cocycle and the cocycle twist $A^\alpha$ is isomorphic to the
lifting $B$ with scalars $a_{i}, a_{ij}$ as above.
\end{theorem}

\begin{proof}
First we show that $\alpha_{\tau}\in Z^{2}_{H}(A,K)$ when $C_{\tau}$
has more than $2$ elements. In this case $r_{i} = 2$ and $q_{i} =
-1$ for every $i \in C_{\tau}$ so that for all $i,j $ with $a_{ij}
\neq0$, $1-g_{i}g_{j} $ is in the centre of $A$. Then by the same
argument as in Lemma \ref{rm:
liftings} (iii), $\gamma\in Z^{2}_{H}(A,K)$ if and only if
$\gamma\in Z^{2}_{H}
(A^{\gamma_{ij}},K)$. From Lemma \ref{rm: liftings}(iv), $\gamma\in Z^{2}%
_{H}(A,K)$ if and only if $\gamma\in Z^{2}_{H} (A^{\gamma_{i }},K)$
and so the fact from Proposition \ref{pr: cocycles commute} that the
$\gamma_{ij}$ and $\gamma_{k}$ commute for any $i,j,k \in C_{\tau}$
implies that $\alpha_{\tau }\in Z^{2}_{H}(A,K)$.

\par Similarly for $\tau\neq\omega$, from the fact that
$1-g_{i}^{r_{i}}$ commutes with all $x_{j}$ and $1-g_{i}g_{j}$
commutes with all $x_{k}$ for $k \neq i,j$, then by the argument in
Lemma \ref{rm: liftings},
$\alpha_{\tau}\in Z^{2}_{H}(A,K)$ if and only if $\alpha_{\tau}\in
Z^{2}_{H}(A^{\omega},K)$. Thus $\alpha\in Z^{2}_{H}(A,K)$.

\par The last statement follows as in  \cite[Proposition 5.9]{ABM},
or by  Proposition \ref{pro: useful3}.
\end{proof}

\begin{remark}
The twisting cocycle $\alpha$ constructed above is not unique.
Cocycles $\gamma$ such that $A^{\gamma}= A$ are discussed in
\cite{Chen} where it is noted that such cocycles form a group. In
\cite{BichonCarnovale} these are called  \textit{lazy cocycles}. So
we have that the twisting cocycle is unique only up to
multiplication by a lazy cocycle,
 that is, $A^\alpha=A^\beta$ if and only if $\beta^{-1}\circ\alpha$ is
a lazy cocycle.

\par Note that in the definition of $\gamma_{i}$ (respectively
$\gamma_{ij}$) it is not required that  $g_{i}^{r_{i}}\neq1$
($g_{i}g_{j}\neq1$). However, if $g_i^{r_i} = 1$ (respectively
$g_ig_j = 1$)  then the cocycle is a lazy one; twisting by such a
cocycle gives the trivial lifting.\end{remark}

\vspace{1mm}

\section{When is the twisting cocycle of the form $\lambda\circ\xi$?}

Recall the motivating question.  For any Hopf algebra $A \cong R
\#_\xi H$, is $\gamma_R^{-1} = \Lambda \circ \xi$   a cocycle for
$R$?  If so, then
 twisting $A$ by this cocycle, extended to $A$, gives a Radford biproduct, as desired.

Let $ R\#_{\xi}H$ as before be a lifting of a quantum plane. In this
section, we present our main results; we  give sufficient conditions
for the composition of the cocycle $\xi$ for the pre-bialgebra and
the total integral $\lambda$ on $H$ to be a cocycle which twists the
Radford biproduct to the given lifting in the setting of
$H=K[\Gamma]$ and $V$ a quantum plane. Recall that the total
integral $\lambda$ on $K[\Gamma]$ is given on generators by
$\lambda(g)=\delta_{g,e},$ the Kronecker delta, where $g\in\Gamma$
and $e$ denotes the unit of $\Gamma$.  The following, together with
the examples of dimension $32$ in \cite{ABM} show that $\lambda
\circ \xi  = \lambda\circ\pi\circ m_{A^{\alpha}}$ is not always
equal to $\alpha$, where $\alpha$ is the cocycle from the previous
section. To construct our examples below it is useful to have an
explicit form of $\alpha$ for  a quantum plane.
\begin{remark}
\label{rem: cocycle}(cf. \cite[Section 5]{ABM}) For $A =
\mathcal{B}(V) \# K[\Gamma]$, $V$ a quantum plane, we summarize the
action of the cocycle $\alpha:= \gamma_{a} \ast \gamma_{1}
\ast\gamma_{2}$ on $A \otimes A$. For $0 < i,k,m,n,t $ we have

\begin{itemize}
\item[(i)] $\alpha(z \otimes1) = \alpha(1 \otimes z) = \varepsilon(z)$ for all
$z \in A$.

\item[(ii)] $\alpha(x_{i}^{n} \otimes x_{i}^{m} ) = \delta_{n+m,r_{i}} a_{i}
.$

\item[(iii)] $\alpha(x_{1}^{m} \otimes x_{2}^{k}) = 0.$

\item[(iv)] $\alpha(x_{1}^{i} \otimes x_{1}^{k} x_{2}^{m}) = 0 = \alpha
(x_{1}^{i}x_{2}^{k} \otimes x_{2}^{m}).$

\item[(v)] $\alpha(x_{2}^{m} \otimes x_{1}^{n}) = \delta_{n,m} (m)!_{q}a^{m}.$

\item[(vi)] $\alpha(x_{2}^{i} \otimes x_{1}^{k}x_{2}^{m}) = \delta
_{i+m,r_{2}+k}\binom{i}{k}_{q}(k)!_{q}a^{k}a_{2} $.

\item[(vii)] $\alpha(x_{1}^{i}x_{2}^{k} \otimes x_{1}^{m}) = \delta
_{i+m,r_{1}+k}\binom{m}{k}_{q}(k)!_{q}a^{k}a_{1} $.

\item[(viii)] $\alpha(x_{1}^{i}x_{2}^{k} \otimes x_{1}^{m} x_{2}^{t}) =
\delta_{i+m,k+t}(i+m-r)!_{q}\binom{k}{r-t}_{q} \binom{m}{r-i}_{q}
q^{it}a^{i+m-r}a_{1}a_{2} $ where $a \neq0$, $r=r_{1}=r_{2}$.

\item[(ix)] $\alpha(x_{1}^{i}x_{2}^{k} \otimes x_{1}^{m} x_{2}^{t}) =
\delta_{i+m,r_{1}}\delta_{k+t,r_{2}} \chi_{1}^{it}(g_{2}) a_{1}a_{2}
$ for $a =0$.
\end{itemize}
\end{remark}

\begin{example}
\label{ex: cocycle not lambda} Let $r,s$ be integers greater than $1$ and let
$\Gamma= C_{rs} \times C_{r} = \langle g \rangle\times\langle h \rangle$ be
the product of two cyclic groups. Let $V = Kx_{1} \oplus Kx_{2}$ be a quantum
plane where $x_{1} \in V_{g}^{\chi} $ and $x_{2} \in V_{g^{-1}h}^{\chi^{-1}}$
with $\chi(g) = q$, $q $ a primitive $r$th root of $1$ and $\chi(h) = q^{2}$.

We verify that $V$ is a quantum plane and that the conditions
(\ref{eq: condnainonzero}) and (\ref{eq: condaijnonzero}) needed to form a
lifting with nonzero scalars $a_{1},a_{2},a$ hold.

\begin{itemize}
\item $\chi(g^{-1}h)\chi^{-1}(g) = q^{-1}q^{2}q^{-1} = 1$;

\item $\chi(g) = q$ and $\chi^{-1}(g^{-1}h) = q q^{-2} = q^{-1}$ with both
$q,q^{-1}$ primitive $r$th roots of unity;

\item $g^{r} \neq1$ and $g^{-r}h^{r} = g^{-r} \neq1$;

\item $g_{1}g_{2} = g g^{-1}h = h \neq1$;

\item $\chi^{r} = \varepsilon$ and by definition $\chi_{1}\chi_{2} =
\varepsilon$.
\end{itemize}

Now let $B$ be a lifting of $A=\mathcal{B}(V)\#K[\Gamma]$ with nonzero scalars
$a_{1},a_{2}$ and with $a$ arbitrary. We know $B=A^{\alpha}$ where the cocycle
$\alpha\in Z_{K[\Gamma]}^{2}(A,K)$ is given in Remark \ref{rem: cocycle}. Thus
$\alpha(x_{1}x_{2}\otimes x_{1}^{r-1}x_{2}^{r-1})=q^{-1}a_{1}a_{2}$.

Let $\lambda$ denote the total integral from the group algebra
$K[\Gamma]$ to $K$. Then
\[
\lambda\pi m_{A^{\alpha}}(x_{1}x_{2}\otimes
x_{1}^{r-1}x_{2}^{r-1})=\lambda \pi
\left[x_{1}\left(qx_{1}x_{2}+a(1-g_{1}g_{2})\right)
x_{1}^{r-2}x_{2}^{r-1}\right].
\]
Continuing in this way, since $\pi(x_{1}^{i}x_{2}^{i})=0$ for $0<i<r$, we
obtain
\begin{align*}
\lambda\pi(q^{r-1}x_{1}^{r}x_{2}^{r})  &  =\lambda\pi \left[ q^{-1}a_{1}%
a_{2}(1-g^{r})(1-g^{-r}h^{r})\right]\\
&  =q^{-1}a_{1}a_{2}\lambda\lbrack1-g^{r}-g^{-r}h^{r}+1]\\
&  =2q^{-1}a_{1}a_{2}=2\alpha(x_{1}x_{2}\otimes x_{1}^{r-1}x_{2}^{r-1}).
\end{align*}
Thus, here, $\lambda\circ\pi\circ m_{A^{\alpha}}\neq\alpha$. \qed

\end{example}

\begin{example}
Let $\Gamma=C_{8}=\langle c\rangle$ be the cyclic group of order
$8$. Let $V=Kx_{1}\oplus Kx_{2}$ be a quantum plane where $x_{1}\in
V_{c}^{\chi}$ and $x_{2}\in V_{c^{5}}^{\chi^{-1}}$ with $\chi(c)=q$,
$q$ a primitive $4$th root of $1$. Verifying that $V$ is a quantum
plane and that (\ref{eq: condnainonzero}) and (\ref{eq:
condaijnonzero}) hold, we compute:

\begin{itemize}
\item $\chi(c^{5})\chi^{-1}(c)=q^{5}q^{-1}=1$;

\item $\chi(c)=q$ and $\chi^{-1}(c^{5})=q^{-5}=q^{-1}$ with both $q,q^{-1}$
primitive $4$th roots of unity;

\item $c^{4}\neq1$ and $(c^{5})^{4}=c^{4}\neq1$;

\item $g_{1}g_{2}=c^{6}\neq1$;

\item $\chi^{4}=\varepsilon$ and by definition $\chi_{1}\chi_{2}=\varepsilon$.
\end{itemize}

Now let $B$ be a lifting of $A=\mathcal{B}(V)\#K[\Gamma]$ with nonzero scalars
$a_{1},a$ and with $a_{2}$ arbitrary. We know $B=A^{\alpha}$ where the cocycle
$\alpha\in Z_{K[\Gamma]}^{2}(A,K)$ is given in Remark \ref{rem: cocycle}. Thus
$\alpha(x_{1}^{3}x_{2}^{2}\otimes x_{1}^{3})=\binom{3}{2}_{q}(2)!_{q}%
a^{2}a_{1}=(3)!_{q}a^{2}a_{1}=(q-1)a^{2}a_{1}$. On the other hand
\begin{align*}
&  \lambda\pi m_{A^{\alpha}}(x_{1}^{3}x_{2}^{2}\otimes x_{1}^{3})=\lambda
\pi \left[x_{1}^{3}x_{2}(x_{2}x_{1})x_{1}^{2}\right]\\
&  =\lambda\pi
\left[x_{1}^{3}x_{2}\left(qx_{1}x_{2}+a(1-c^{6})\right)x_{1}^{2}\right]=q\lambda
\pi\left[x_{1}^{3}x_{2}(x_{1}x_{2})x_{1}^{2}\right]+a\lambda\pi\left[x_{1}^{3}x_{2}%
(1-c^{6})x_{1}^{2}\right]
\end{align*}

We note that $c^{6} x_{i} = - x_{i} c^{6}$ and $c^{6}$ and $x_{i}^{2}$
commute. Computing the first summand we obtain:
\begin{align*}
&   q\lambda\pi\left[  x_{1}^{3}\left(
qx_{1}x_{2}+a(1-c^{6})\right) \left(
qx_{1}x_{2}+a(1-c^{6})\right)  x_{1}\right]  \\
&   = q\lambda\pi\left [q^{2} x_{1}^{4} (x_{2}x_{1})^{2} + 2q
ax_{1}^{4}
(1+c^{6})x_{2}x_{1} + a^{2}x_{1}^{4}(1+c^{6})^{2}\right]\\
&   = qa_{1} \lambda\left[(1-c^{4}) \pi\left(-x_{2}x_{1}x_{2}x_{1} + 2qa(1+c^{6}%
)x_{2}x_{1} + a^{2}(1+c^{6})^{2}\right)\right]\\
&   = qa_{1} \lambda\left[(1-c^{4})\left(-a^{2}(1-c^{6})^{2} +
2qa^{2}(1+c^{6})(1-c^{6})
+ a^{2}(1+c^{6})^{2}\right)\right]\\
&   = qa_{1}a^{2}\lambda\left[ 4c^{6} - 4c^{2} +4q(1-c^{4})\right]\\
&  = 4q^{2}a_{1}a^{2} = -4a_{1}a^{2}.
\end{align*}
Computing the last summand we obtain:
\begin{align*}
  a\lambda\pi\left[x_{1}^{3}x_{2}x_{1}^{2}(1-c^{6})\right]  &  = a\lambda \left[\pi
\left(x_{1}^{3}(x_{2}x_{1})x_{1}\right)(1-c^{6})\right]\\
&  =a\lambda \left[\pi\left(x_{1}^{3}(qx_{1}x_{2}+a(1-c^{6}))x_{1}\right)(1-c^{6})\right]\\
&  =a\lambda \left[\lbrack q\pi(x_{1}^{3}x_{1}x_{2}x_{1})+a\pi(x_{1}%
^{3}(1-c^{6})x_{1})](1-c^{6})\right]\\
&  = a \lambda \left[q \pi(x_{1}^{4} x_{2} x_{1})(1-c^{6}) + aa_{1}(1-c^{4}%
)(1+c^{6})(1-c^{6})\right]\\
&  = a_{1}a^{2} \lambda \left[q(1-c^{4})(1-c^{6})^{2} + (1-c^{4})(1+c^{6}%
)(1-c^{6})\right]\\
&  = a_{1}a^{2} \lambda \left[2q(c^{2} - c^{6}) + 2(1-c^{4})\right]
= 2a_{1}a^{2}.
\end{align*}

Thus $\lambda\circ\pi\circ m_{A^{\alpha}}(x_{1}^{3}x_{2}^{2}\otimes x_{1}^{3}
)= -2a_{1}a^{2} \neq\alpha(x_{1}^{3}x_{2}^{2}\otimes x_{1}^{3} )$.
\end{example}

Of course the examples above only show that $\lambda\circ\pi\circ
m_{A^{\alpha}}$ is not the cocycle $\alpha$; it is not known whether
or not $\lambda\circ\pi\circ m_{A^{\alpha}}$ is a cocycle.  Note
also that in \cite{ABM-gauge} we investigate deformations of Hopf
algebras with the dual Chevalley property which naturally involves
$\lambda\circ\pi\circ m_{A^{\alpha}}$ just as a gauge transformation
and not necessarily a cocycle.

Next we show that in many cases
$(g_{1}g_{2})^{r}\neq1$ is enough to ensure that a twisting cocycle
for the lifting of a quantum plane is the composite of the integral
and the cocycle for the pre-bialgebra.

\begin{theorem}
\label{thm: lambdaxi}Let $V=Kx_{1}\oplus Kx_{2}$ be a quantum plane with
$x_{i}\in V_{g_{i}}^{\chi_{i}}$. Let $R=\mathcal{B}(V)$, and let $A:=R\#H$.
Let $B:=A(a_{1},a_{2},a)$ be a lifting of $A$. Let $\alpha$ be the cocycle
that twists $A$ into $B$, i.e. $A^{\alpha}=B$. Then $B=R^{\alpha_{R}%
}\#_{\varepsilon_{\alpha_{R}}}H$. Suppose that one of the following conditions holds:

\begin{itemize}
\item[(i)] $\alpha= \gamma_{i}$ for $i = 1 $ or $2$;

\item[(ii)] $\alpha= \gamma_{1} \ast\gamma_{2}$ and $g_{1}^{r_{1}}g_{2}%
^{r_{2}} \neq1$;

\item[(iii)] The parameter $a:=a_{12}$ is nonzero, $r$ is odd or $r= 2$ and
$(g_{1}g_{2})^{r} \neq1$ where $r:= r_{1} = r_{2}$;

\item[(iv)] The parameter $a$ is nonzero, $r = 2r^{\prime}>2$ and $(g_{1}%
g_{2})^{tr^{\prime}} \neq1$ where $t \in\{1,2,3 \}$ and $(g_{1}g_{2}%
)^{r^{\prime}} g_{2}^{ r} \neq1 \neq(g_{1})^{r}(g_{1}g_{2})^{r^{\prime}}$.
\end{itemize}

Then for $\lambda$ the total integral on the group algebra $H = K[\Gamma]$,
\[
\lambda\circ\varepsilon_{\alpha_{R}} = \alpha_{R}.
\]

\end{theorem}

\begin{proof}
By \cite[4.11]{ABM}, $B = A^{\alpha}= R^{\alpha_{R}} \#_{ \varepsilon
_{\alpha_{R}}} K[\Gamma]$ and $\pi m_{A^{\alpha}} = \varepsilon_{\alpha_{R}} =
u_{H} \alpha_{R} \ast\varepsilon\ast(H \otimes\alpha^{-1}_{R})\rho_{R \otimes
R}$ so that
\[
\lambda\circ\varepsilon_{\alpha_{R}} = \alpha_{R} \ast(\lambda\otimes
\alpha^{-1}_{R})\rho_{R \otimes R}.
\]
Thus since $\alpha_{R}$ is convolution invertible, $\lambda\circ
\varepsilon_{\alpha_{R}} = \alpha_{R}$ if and only if $(\lambda\otimes
\alpha^{-1}_{R})\rho_{R \otimes R} = \varepsilon_{R \otimes R}$.

Suppose first that $\alpha= \gamma_{i}$ and assume that $a_{i}
\neq0$. Then $\alpha^{-1}= \gamma_{-a_{i}}$, the cocycle defined
exactly as $\gamma_{i}$ is but with $a_{i}$ replaced by $-a_{i}$,
and for $0 < n < r_{i}$, $(\lambda \otimes\alpha^{-1})\rho_{R
\otimes R}(x_{i}^{n} \otimes x_{i}^{r_{i} -n})=
\lambda(g_{i}^{r_{i}})(-a_{i}) = 0$ since if $a_{i} \neq0$,
$g_{i}^{r_{i}} \neq 1$. It is clear that
$(\lambda\otimes\alpha^{-1})\rho_{R \otimes R} = \varepsilon_{R
\otimes R}$ on all other elements of $R \otimes R$ and so the
statement holds if $\alpha= \gamma_{i}$.

Now suppose that $\alpha=\gamma_{1}\ast\gamma_{2}$ with
$a_{1},a_{2}$ nonzero. Then $ \alpha^{-1} = \gamma_2^{-1} \ast
\gamma_1^{-1}$  so that again $\alpha$ and $\alpha^{-1}$ are nonzero
on the same elements. By Remark \ref{rem: cocycle}, $\alpha_{R}$ can
be different from $\varepsilon$ only on elements of the form
$z=x_{i}^{n}\otimes x_{i}^{r_{i}-n}$ or $z=x_{1}^{n}x_{2}^{m}\otimes
x_{1}^{r_{1}-n}x_{2}^{r_{2}-m}$. In the first case
$(\lambda\otimes\alpha_{R}^{-1})\rho_{R\otimes R}(z)=0$ by the
argument above
and in the second case, this holds by the assumption that $g_{1}^{r_{1}}%
g_{2}^{r_{2}}\neq1$.

Finally assume that $a\neq0$. Then $\chi_{2}=\chi_{1}^{-1}$. Suppose that
$\chi_{1}(g_{1})=q$, a primitive $r$th root of unity. Then $1=\chi_{1}%
(g_{2})\chi_{2}(g_{1})$ forces $\chi_{1}(g_{2})=q$ and $\chi_{2}(g_{2}%
)^{-1}=q$ as in Lemma \ref{re: chi1chi2varepsilon}. By Lemma
\ref{lm: qplbasic lemma}, $\alpha^{-1}(x_{1}^{i}x_{2}^{j}\otimes x_{1}%
^{k}x_{2}^{l})=\varepsilon_{R\otimes R}(x_{1}^{i}x_{2}^{j}\otimes x_{1}%
^{k}x_{2}^{l})$ unless $q^{i+k}=q^{j+l}$. Clearly $(\lambda\otimes\alpha
_{R}^{-1})\rho_{R\otimes R}(1_{R}\otimes1_{R})=1=\varepsilon_{R\otimes
R}(1_{R}\otimes1_{R})$ and it remains to show that $(\lambda\otimes\alpha
_{R}^{-1})\rho_{R\otimes R}(x_{1}^{i}x_{2}^{j}\otimes x_{1}^{k}x_{2}^{l})=0$
for $i+k\equiv j+l$ mod $r$ with $i+j+k+l>0$.

Suppose that $i + k = 0$ so that $i = k = 0$. Then $j + l = r$ and
$(\lambda\otimes\alpha^{-1}_{R})\rho_{R \otimes R}( x_{2}^{j} \otimes
x_{2}^{l}) = \lambda(g_{2}^{r})\alpha^{-1}( x_{2}^{j} \otimes x_{2}^{l})$. If
$a_{2} \neq0$ then $g_{2}^{r} \neq1$, and $\lambda(g_{2}^{r}) = 0$. If $a_{2}
= 0$, then $\alpha^{-1}( x_{2}^{j} \otimes x_{2}^{l}) = - a_{2} = 0 $. The
case where $i+k = r$ and $j+l = 0$ is similar.

Now consider $0 < i+k <2r$, $0 < j+l <2r$ and $j+l \equiv i+k $ mod $r$. Then
\[
(\lambda\otimes\alpha^{-1}_{R})\rho_{R \otimes R}(x_{1}^{i}x_{2}^{j} \otimes
x_{1}^{k} x_{2}^{l}) = \lambda(g_{1}^{i+k} g_{2}^{ j+l })\alpha^{-1}(x_{1}%
^{i}x_{2}^{j} \otimes x_{1}^{k} x_{2}^{l}).
\]
This expression is $0$ unless $g_{1}^{i+k} g_{2}^{ j+l } = 1$. If $g_{1}^{i+l}
g_{2}^{ j+l } = 1$ then $1 = \chi_{1}(g_{1}^{i+k} g_{2}^{ j+l }) = q^{2(i+k)}$
so that $r$ divides $2(i+k)$. If $r$ is odd, then $r$ must divide $i+k$ and
thus $i+k = j+l = r$. If $(g_{1}g_{2})^{r} \neq1$, then $(\lambda\otimes
\alpha^{-1}_{R})\rho_{R \otimes R}(x_{1}^{i}x_{2}^{j} \otimes x_{1}^{k}
x_{2}^{l}) =0$.

If $r = 2$, then the nonzero possibilities for $i+k$ and $j+l$ are $1$ and
$2$. Thus $i+k = j+l$. If $i+k=j+l = 1$ then since $g_{1}g_{2} \neq1$, the
statement holds. If $i+k=j+l = 2= r$, then $(g_{1}g_{2})^{2} \neq1$ guarantees
that $(\lambda\otimes\alpha^{-1}_{R})\rho_{R \otimes R}(x_{1}^{i}x_{2}^{j}
\otimes x_{1}^{k} x_{2}^{l}) =0$.

Suppose $r = 2r^{\prime}>2$. Then $q^{2(i+k)} = 1$ implies that $r^{\prime}$
divides $i+k$ so that $i+k = j+l \in\{ r^{\prime}, 2r^{\prime}= r, 3r^{\prime
}\}$ or else $i+k = r^{\prime}$, $j+l = 3r^{\prime}$ or $i+k = 3r^{\prime},
j+l = r^{\prime}$. The conditions in (iv) imply that in all cases
$(\lambda\otimes\alpha^{-1}_{R})\rho_{R \otimes R}(x_{1}^{i}x_{2}^{j} \otimes
x_{1}^{k} x_{2}^{l}) =0$.
\end{proof}

The above theorem can be used to find sufficient conditions for $\alpha=
\lambda\circ\varepsilon_{\alpha_{R}}$ for various quantum linear spaces,
although general statements become unwieldy. The next corollary is a simple
extension of (ii) in the theorem.

\begin{corollary}
Let $W = \oplus_{i=1}^{\theta}Kx_{i}$ be a quantum linear space and
$A:= \mathcal{B}(W) \# K[\Gamma] = R \# H$. Let $\alpha\in
Z^{2}_{H}(A,K)$ be of the form $\alpha= \prod_{i \in I} \gamma_{i}$
with $I \subseteq\{1, \ldots, \theta\}$ and $\gamma_i \neq
\varepsilon$. If $g_{j_{1}}^{r_{j_{1}}} \ldots g_{j_{t}}^{r_{j_{t}}}
\neq1 $ for $\{ j_{1}, \ldots, j_{t} \}$ any nonempty subset of $I$,
then $\lambda\circ \pi\ \circ m_{B} = \alpha_{R}$.
\end{corollary}

\begin{proof}
The proof is the same as the proof of the sufficiency of (i) or (ii) in the
above theorem.
\end{proof}

\appendix
\section{Basis preserving multiplications}
 Throughout this appendix, as in the paper,  $\Gamma$ is a finite abelian group,
 $H$ is the group algebra $K[\Gamma]$,
 and $W = \sum_{i=1}^\theta Kx_i$ is a quantum linear space with $x_i
 \in W^{\chi_i}_{g_i}$. As well, throughout this section $A=R\#H$ with
$R=\mathcal{B}(W)$. When needed for emphasis we write $\cdot_A$ for
multiplication in $A$; if the  context is clear we write
multiplication in $A$ as concatenation.

\begin{definition}
   Consider   a map $\mu:A\otimes
A\rightarrow A$ and denote by $x\cdot_{\mu}y$ the product
$\mu(x\otimes y) $ for all $x,y\in A$.      We say that $\mu$
\textbf{preserves the basis of }$A$ whenever $\mu$ is $H$-bilinear (where $A\otimes A$ is an $H$-bimodule via $h(a\otimes b)k = ah
\otimes bk$ and $A$ is an $H$-bimodule via $h(a)k = hak$),
$H$-balanced (i.e. $\mu(ah\otimes b)=\mu(a\otimes hb)$), associative and unitary  with respect to $1_{A}$  and
 for all $0\leq n_{i}\leq
r_{i}-1,g\in\Gamma$
\begin{equation}
x_{1}^{n_{1}}\cdots x_{\theta}^{n_{\theta}}g=x_{1}^{\cdot_{\mu}n_{1}}%
\cdot_{\mu}\cdots\cdot_{\mu}x_{\theta}^{\cdot_{\mu}n_{\theta}}\cdot_{\mu}g.
\label{form:presbasis}%
\end{equation}

\end{definition}

 Note that $H$-bilinearity of $\mu$ implies that for $h,l \in H$, $h
\cdot_\mu l = hl$. Let $X_\mu$ denote
$x_{1}^{\cdot_{\mu}n_{1}}\cdot_{\mu}\cdots\cdot_{\mu}x_{\theta}^{\cdot_{\mu
}n_{\theta}}$.  Since   for all $ 0\leq n_{i}\leq
r_{i}-1,g\in\Gamma$,
\begin{gather*}
X_\mu \cdot_{\mu}g=   X_\mu \cdot_{\mu }\left(  1_{A}g\right)
 =\left(  X_\mu  \cdot_{\mu}1_{A}\right) g=X_\mu  g
\end{gather*}
then  (\ref{form:presbasis}) is equivalent to
\begin{equation} \label{form:presbasis no g}
x_{1}^{n_{1}}\cdots
x_{\theta}^{n_{\theta}}=x_{1}^{\cdot_{\mu}n_{1}}\cdot
_{\mu}\cdots\cdot_{\mu}x_{\theta}^{\cdot_{\mu}n_{\theta}}.
\end{equation}

The proof of the next lemma is immediate.

\begin{lemma}
\label{lem: basis}
 Let $\mu: A \otimes  A \rightarrow A$ be $H$-bilinear,
$H$-balanced, associative and unitary  with respect to $1_{A}$.
 Then $\mu$ preserves the basis of $A$ if and only if the following
conditions hold: \begin{itemize} \item[(1)] $x_{i}^{\cdot_{A}m}$
$\cdot_{\mu}x_{i} =x_{i}^{\cdot_{A}\left( m+1\right)  }$ for all $i$
and for all $0 \leq m < r_{i}-1$;\\
\item[(2)] $\left(
x_{1}^{\cdot_{A}n_{1}}\cdot_{A}\cdots\cdot_{A}x_{s-1}^{\cdot
_{A}n_{s-1}}\right)  \cdot_{\mu}x_{s}^{\cdot_{A}n_{s}}=x_{1}^{\cdot_{A}n_{1}%
}\cdot_{A}\cdots\cdot_{A}x_{s-1}^{\cdot_{A}n_{s-1}}\cdot_{A}x_{s}^{\cdot
_{A}n_{s}}$ for all $1\leq s\leq\theta$ and for all $0\leq n_{i}\leq r_{i}-1.$
\end{itemize}
\end{lemma}

\begin{lemma}
\label{lem:Relations} Let $\mu: A \otimes A \rightarrow A$ be a map
which
   preserves the basis of $A.$
Suppose that for each $i$ there exist $m_{i}, n_{i}$ with $0\leq
m_{i},n_{i}\leq r_{i}-1,$ $m_{i}+n_{i}=r_{i}$, such that
$x_{i}^{\cdot_{A}m_{i}}\cdot_{\mu }x_{i}^{\cdot_{A}n_{i}}\in H$ and
set $h_{i}:=x_{i}^{\cdot_{A}m_{i}}\cdot
_{\mu}x_{i}^{\cdot_{A}n_{i}}.$ Also suppose for all $1\leq
j<i\leq\theta$ then $x_{i}\cdot_{\mu}x_{j}-\chi_{j}\left(
g_{i}\right) x_{j}\cdot_{\mu }x_{i}\in H$ and set
$h_{i,j}:=x_{i}\cdot_{\mu}x_{j}-\chi_{j}\left( g_{i}\right)
x_{j}\cdot_{\mu}x_{i}$. Then  $\mu$ is completely determined by
$\left(  h_{i}\right)  _{1\leq i\leq\theta},\left( h_{u,v}\right)
_{1\leq u,v\leq\theta}$.
\end{lemma}

\begin{proof}   Let $X=x_{1}^{  m_{1}}  \cdots
x_{\theta}^{
m_{\theta}}=x_{1}^{\cdot_{\mu}m_{1}}\cdot_{\mu}\cdots\cdot_{\mu}x_{\theta
}^{\cdot_{\mu}m_{\theta}}$, $ Y= x_1^{n_1} \cdots
x_\theta^{n_\theta} =x_{1
}^{\cdot_{\mu}n_{1}}\cdots\cdot_{\mu}x_{\theta}^{\cdot_{\mu
}n_{\theta}}  $.
 Since $A$ has basis $\{ Xg = X \cdot_\mu g| X \text{ as above }, g \in \Gamma
 \}$,
 $\mu$ is determined by the elements $(X \cdot_\mu g) \cdot_\mu (Y
\cdot_\mu h)$ for $g,h \in \Gamma$ and
\begin{align*}
&  (X \cdot_\mu g) \cdot_\mu (Y \cdot_\mu h) =X \cdot_{\mu}\left(
g\cdot_{\mu}\left( Y \cdot_{\mu}h\right) \right)
   \overset{\mu\text{ left }H\text{-lin}}{=} X \cdot_{\mu}
    \left(
 g \left( Y_{\mu} \cdot_{\mu} h \right)  \right)
\\
&  = X \cdot_{\mu}\left(  g\left( Y h\right)  \right)
 =
X \cdot_{\mu}\left(  Y gh\right) \chi \left( g\right)
  =
   \left(  X \right)  \cdot_{\mu}\left( Y \right)
\cdot_{\mu}\left(  gh\right)     \chi    \left( g\right),
\end{align*}
where $\chi =
\chi_{1}^{n_{1}}\ast\cdots\ast\chi_{\theta}^{n_{\theta}  }$ so that
$\mu$ is determined by the elements $
X \cdot_\mu Y $.  It is straightforward to prove by induction on $n\geq1$ that%
\begin{equation}\label{eqn: comm ij}
x_{i}^{\cdot_{\mu}m}\cdot_{\mu}x_{j}^{\cdot_{\mu}n}=\sum_{t=0}^{\min\left(
m,n\right)  }x_{j}^{\cdot_{\mu}\left(  n-t\right)  }\cdot_{\mu}x_{i}%
^{\cdot_{\mu}\left(  m-t\right)
}\cdot_{\mu}\beta_{m,n,t}^{i,j}\left( h_{i,j}\right)
\end{equation}
where $\beta_{m,n,t}^{i,j}\left(  h_{i,j}\right) \in H  $ only
depends on $g_{1},\ldots,g_{\theta},\chi_{1},\ldots,\chi_{\theta}$
and $h_{i,j}$. Note that (\ref{eqn: comm ij}) holds also for $m=0$
or $n=0$.

Thus
\begin{align*}
&  X  \cdot_{\mu}Y    =\sum_{0\leq t_{1}\leq
m_{1}+n_{1},\cdots,0\leq t_{\theta}\leq m_{\theta
}+n_{\theta}}x_{1}^{\cdot_{\mu}t_{1}}\cdot_{\mu}\cdots\cdot_{\mu}x_{\theta
}^{\cdot_{\mu}t_{\theta}}\cdot_{\mu}\gamma_{t_{1},\cdots,t_{\theta}}\left(
\left(  h_{u,v}\right)  _{1\leq u,v\leq\theta}\right)
\end{align*}
where $\gamma_{t_{1},\cdots,t_{\theta}}\left(  \left(  h_{u,v}\right)  _{1\leq
u,v\leq\theta}\right)  \in H$ only depends on $g_{1},\ldots,g_{\theta}%
,\chi_{1},\ldots,\chi_{\theta}$ and $\left(  h_{u,v}\right)  _{1\leq
u,v\leq\theta}$.

Since $x_i^n = x_i^{\cdot_\mu n}$ for $n<r_i$ and  $x_{i}^{\cdot_{\mu}r_{i}}=x_{i}^{\cdot_{\mu}m_{i}}\cdot_{\mu}%
x_{i}^{\cdot_{\mu}n_{i}}=x_{i}^{m_{i}}\cdot_{\mu}x_{i}^{n_{i}}=h_{i}\in
H$, it remains to consider
  $x_{i}^{\cdot_{\mu}n}$ where $n >r_i$. Suppose that $n = qr_i +
  R$. Then
\[
x_{i}^{\cdot_{\mu}n}=x_{i}^{\cdot_{\mu}\left[ R + qr_i   \right]
}=x_{i}^{\cdot_{\mu}R  }\cdot _{\mu}\left(
x_{i}^{\cdot_{\mu}r_{i}}\right) ^{\cdot_{\mu}q   }=x_{i}^{R
}\cdot_{\mu}h_{i} ^{ q   } = x_{i}^{R } h_{i} ^{ q   }
\]
and the statement is proved. \end{proof}

\vspace{1mm}

\begin{proposition}
\label{pro:mu}For $W,H,R,A$ as defined throughout this appendix, let
$B:=A(a_{i},a_{ij}|1\leq i,j\leq\theta)$ be a lifting of $A$ where
$a_{i},a_{ij}$ are  scalars, not necessarily nonzero.

\par Let $\varphi:A\rightarrow B$ map the basis of $A$ to the basis of $B$ by $\varphi(x_{1}^{\cdot
_{A}n_{1}}\cdot_{A}\cdots\cdot_{A}x_{\theta}^{\cdot_{A}n_{\theta}}\cdot_{A}g)
= x_{1}^{\cdot_{B}n_{1}}\cdot_{B}\cdots
\cdot_{B}x_{\theta}^{\cdot_{B}n_{\theta}}\cdot_{B}g$. Then $\varphi$
is an $H$-bilinear coalgebra isomorphism. Let $\varsigma:A\otimes
A\rightarrow A$ be defined by $\varsigma:=\varphi^{-1}m_{B}\left(
\varphi\otimes\varphi\right)$. Then $\varsigma$  preserves the basis
of $A$.    If we regard $A$ as an algebra through $\varsigma$,
$\varphi$ is a bialgebra isomorphism.

\par Let $\mu: A \otimes A \rightarrow A$ be a map which preserves
the basis of $A.$ Suppose that for each $i$ there exist $m_{i},
n_{i}$ with
$0\leq m_{i},n_{i}\leq r_{i}-1,$ $m_{i}+n_{i}=r_{i}$, such that $x_{i}%
^{\cdot_{A}m_{i}}\cdot_{\mu}x_{i}^{\cdot_{A}n_{i}}=a_{i}(1_{A}-g_{i}^{r_{i}})$
and that $x_{i}\cdot_{\mu}x_{j}-\chi_{j}\left(  g_{i}\right)
x_{j}\cdot_{\mu }x_{i}=a_{ij}(1-g_{i}g_{j})$ for all $1\leq
j<i\leq\theta.$ Then  $\mu=\varsigma.$
\end{proposition}

\begin{proof}
 Since the coalgebra structures of
$A$ and $B$ are defined in the same way with respect to the
corresponding basis, $\varphi$ is a bijective coalgebra map.
Furthermore, $\varphi$ is $H$-bilinear since for $X_A:=
x_{1}^{\cdot_{A}n_{1}}\cdot_{A}\cdots
\cdot_{A}x_{\theta}^{\cdot_{A}n_{\theta}}$, $X_B:=
x_{1}^{\cdot_{B}n_{1}}\cdot_{B}\cdots\cdot_{B}x_{\theta
}^{\cdot_{B}n_{\theta}}  $, $\chi:=
\chi_{1}^{n_{1}}\ast\cdots\ast\chi_{\theta }^{n_{\theta}}$,
\begin{align*}
&  \varphi\left[  h\cdot_{A}\left( X_A \cdot_{A}g\right)
\cdot_{A}l\right]
   =
  \varphi\left[ \chi \left(  h\right)  \left(  X_A \cdot_{A}g\right)
\cdot_{A}h\cdot_{A}l\right]
\\
 & =
  \varphi\left[  \chi  \left(
h\right)  \left( X_A \cdot_{A}g\right) \cdot_{A}\left(  hl\right)
\right]
  = \chi \left(  h\right)  \left( X_B \cdot_{B}g\right)
\cdot_{B}\left( hl\right)
\\
  & =
    \chi \left(  h \right)
     \left( X_B \cdot_{B}g \right)
  \cdot_{B}\left( h \cdot_{B}l \right)
  = h\cdot_{B}\left( X_B \cdot_{B}g\right)  \cdot_{B}l =
  h\cdot_{B}\varphi\left(  X_A \cdot_{A}g\right)  \cdot_{B}l.
\end{align*}

Since $m_{B}$   preserves the basis of $B $, then $\varsigma$
preserves the basis of $A$.

For all $0\leq m \leq r_{i}-1,$ $ n=r_{i}-m,$ we have that $
x_{i}^{\cdot_{A}m}\cdot_{\varsigma}x_{i}^{\cdot_{A}n}
=a_{i}(1_{A}-g_{i}^{r_{i}})$ since
\begin{align*}
\varphi^{-1}\left[
\varphi\left(  x_{i}^{\cdot_{A}m}\right)  \cdot_{B}\varphi\left(  x_{i}%
^{\cdot_{A}n}\right)  \right]  =\varphi^{-1}\left(  x_{i}^{\cdot_{B}m}%
\cdot_{B}x_{i}^{\cdot_{B}n}\right)
 =\varphi^{-1}\left(  x_{i}^{\cdot_{B}r_{i}}\right)  =\varphi^{-1}\left(
a_{i}(1_{B}-g_{i}^{r_{i}})\right)  .
\end{align*}
For $1\leq j<i\leq\theta$, we have%
\begin{gather*}
x_{i}\cdot_{\varsigma}x_{j}=\varphi^{-1}\left[  \varphi\left(  x_{i}\right)
\cdot_{B}\varphi\left(  x_{j}\right)  \right]  =\varphi^{-1}\left(  x_{i}%
\cdot_{B}x_{j}\right) \\
=\varphi^{-1}\left(  \chi_{j}\left(  g_{i}\right)  x_{j}\cdot_{B}x_{i}%
+a_{ij}(1_{B}-g_{i}g_{j})\right)  =\chi_{j}\left(  g_{i}\right)  x_{j}%
\cdot_{A}x_{i}+a_{ij}(1_{A}-g_{i}g_{j})
\end{gather*}
and%
\[
x_{j}\cdot_{\varsigma}x_{i}=\varphi^{-1}\left[  \varphi\left(  x_{j}\right)
\cdot_{B}\varphi\left(  x_{i}\right)  \right]  =\varphi^{-1}\left(  x_{j}%
\cdot_{B}x_{i}\right)  =x_{j}\cdot_{A}x_{i}%
\]
so that $x_{i}\cdot_{\varsigma}x_{j}-\chi_{j}\left(  g_{i}\right)  x_{j}%
\cdot_{\varsigma}x_{i}=a_{ij}(1_{A}-g_{i}g_{j}).$  By Lemma
\ref{lem:Relations}, $\mu$ is now completely determined.
\end{proof}

\begin{proposition}
\label{pro:gamma}
 Let $B:=A(a_{i},a_{ij}|1\leq
i,j\leq\theta)$ be an arbitrary lifting of $A.$ For $1 \leq s <
\theta$, let $A_{s}$ be the subalgebra of $A$ generated by
$x_{1},\ldots,x_{s}  ,g_{1},\ldots,g_{s}$. Then $A_{s}$ is generated
by elements of the form
$hx_{1}^{\cdot_{A}n_{1}}\cdot_{A}\cdots\cdot_{A}x_{s}^{\cdot_{A}n_{s}}$
where $0\leq n_{i}\leq r_{i}-1,1\leq
i\leq  s$. Let $\gamma \in Z_{H}
^{2}(A,K)\ $  and suppose that for all $1 \leq i \leq \theta$, $0
\leq n_i,m_i \leq r_{i}-1$,
\begin{equation}
\gamma\left(  x_{i}^{\cdot_{A}n_i}\otimes
x_{i}^{\cdot_{A}m_i}\right) =\varepsilon_{A}\left(
x_{i}^{\cdot_{A}n_i}\right) \varepsilon_{A}\left(
x_{i}^{\cdot_{A}m_i}\right)    \text{ if }n_i+m_i\leq r_{i}-1,
\label{form:gamma1}%
\end{equation}%
\begin{equation}
\gamma\left(  z \otimes x_{s}^{\cdot_{A}n_{s}}\right)
=\varepsilon_{A}\left( z \right)  \varepsilon_{A}\left(
x_{s}^{\cdot_{A}n_{s}}\right) \text{ for all  } z \in A_{s-1} \text{
and
}1 <  s\leq\theta, \label{form:gamma2}%
\end{equation}%
\begin{equation}
\gamma\left(  x_{i}^{\cdot_{A}(r_{i}-1)}\otimes x_{i}\right)  =a_{i}, \label{form:gamma3}%
\end{equation}%
\begin{equation}
     \gamma \left(  x_{i} \otimes x_{j} \right)
   =a_{ij},
  \text{ for all } 1 \leq j<i \leq  \theta. \label{form:gamma4}
\end{equation}
Then
$A^{\gamma }\cong B.$
\end{proposition}

\begin{proof}  First we note that $(\ref{form:gamma1})$ and
$(\ref{form:gamma2})$ also hold for $\gamma^{-1}$.   To see this,
apply $\gamma \ast \gamma^{-1} = m_K(\varepsilon \otimes
\varepsilon)$ to elements of the form $x_{i}^{\cdot_{A}n_i}\otimes
x_{i}^{\cdot_{A}m_i}$ or $z \otimes x_s^{\cdot_An_s}$ where $z \in
A_{s-1}$.

\par Also, applying $\gamma \ast \gamma^{-1} = m_K(\varepsilon \otimes
\varepsilon)$ to elements of the form $z \otimes x_{i} $ where $z
\in A$, we see immediately that
\begin{equation}
\gamma^{-1}\left(  z\otimes x_{i}\right)  =-\gamma\left(  z\otimes
x_{i}\right)  ,\text{ for } 1  \leq i\leq\theta.\label{form:gamma-1}%
\end{equation}

  Let $\mu:=m_{A^{\gamma}}=\gamma\ast m_{A}\ast\gamma^{-1}.$
First we
 check that $\mu$ preserves the basis
of $A$. Let $i \in \{1, \ldots, \theta \}$, and write $x$ for $x_i$,
$g$ for $g_i$, $q$ for $q_i$, $r $ for $r_i$, $a$ for $a_i$. For $1
\leq m+1 \leq r-1$,
\begin{eqnarray*}
 && x^{  m}\cdot_{\mu}x
 \overset{(\ref{form: Deltasquare})}{ = }  \gamma( x^{  m}_{(1)} \otimes g) x^{  m}_{(2)}g
 \gamma^{-1}(x^{  m}_{(3)} \otimes x) +    \gamma( x^{  m}_{(1)} \otimes g)
 x^{
 m}_{(2)}x
 \gamma^{-1}(x^{  m}_{(3)} \otimes 1) \\ && +    \gamma( x^{  m}_{(1)} \otimes x)
 x^{
 m}_{(2)}
 \gamma^{-1}(x^{  m}_{(3)} \otimes 1)
 \overset{(\ref{form:gamma3}),(\ref{form:gamma-1})}{=}   x^{  m+1}+
 \delta_{m+1,r }[ a (1-g^{r })]    = x^{m+1},
\end{eqnarray*}
and Lemma \ref{lem: basis}(1) holds. A similar argument   shows that
Lemma \ref{lem: basis}(2) holds and so $\mu$ preserves the basis of
$A$.  Now, again denoting $x:=x_i$, etc,
\begin{eqnarray*}
&& x^{ \left(  r -1\right)  }\cdot_{\mu}x
\overset{(\ref{form:gamma1}) \text{ for } \gamma, \gamma^{-1}}{=}
 \gamma( x^{    r-1      }\otimes
x )  1_A + \gamma( g^{    r -1 }\otimes g )x ^{ r }   + \gamma^{-1}(
x^{   r -1  }\otimes
x )  g^{r } \\
&& \overset{(\ref{form:gamma3}), (\ref{form:gamma-1})}{=} -g ^{r }a
+a 1_{A}=a \left(  1-g ^{r }\right).
\end{eqnarray*}

For all $a,b\in\{1,\ldots,\theta\},$
\begin{align*}
x_{a}\cdot_{\mu}x_{b}    \overset{(\ref{form:
Deltasquare}),(\ref{form:gamma-1}) }{=}
 -g_{a}g_{b}\gamma\left(  x_{a}\otimes x_{b}\right)  +x_{a}x_{b}%
+\gamma\left(  x_{a}\otimes x_{b}\right)  1_{A}
   =x_{a}x_{b}+\gamma\left(  x_{a}\otimes x_{b}\right)  \left(  1_{A}%
-g_{a}g_{b}\right).
\end{align*}
Therefore for all $1\leq j<i\leq\theta$, since $x_ix_j =
\chi_j(g_i)x_jx_i$ in $A$,
\begin{align*}
&  x_{i}\cdot_{\mu}x_{j}-\chi_{j}\left(  g_{i}\right)  x_{j}\cdot_{\mu}x_{i}\\
&  =\left[  x_{i}\cdot_{A}x_{j}+\gamma\left(  x_{i}\otimes x_{j}\right)
\left(  1_{A}-g_{i}g_{j}\right)  \right]  -\chi_{j}\left(  g_{i}\right)
\left[  x_{j}\cdot_{A}x_{i}+\gamma\left(  x_{j}\otimes x_{i}\right)  \left(
1_{A}-g_{j}g_{i}\right)  \right]
\\
&  =\gamma\left(  x_{i}\otimes x_{j}\right)  \left(  1_{A}-g_{i}g_{j}\right)
-\chi_{j}\left(  g_{i}\right)  \gamma\left(  x_{j}\otimes x_{i}\right)
\left(  1_{A}-g_{j}g_{i}\right)
\\
&  =\left[  \gamma\left(  x_{i}\otimes x_{j}\right)  -\chi_{j}\left(
g_{i}\right)  \gamma\left(  x_{j}\otimes x_{i}\right)  \right]  (1-g_{i}%
g_{j})
\\
&  \overset{(\ref{form:gamma2}),(\ref{form:gamma4})}{=}
a_{ij}(1-g_{i}g_{j}).
\end{align*}
 The statement now
follows from Proposition \ref{pro:mu}.
\end{proof}

\begin{proposition}
\label{pro: useful3}  Let $B:=A(a_{i},a_{ij}|1\leq i,j\leq\theta)$
be an arbitrary lifting of $A$ as in the previous proposition.  Let
$\alpha \in Z^2_H(A,K)$ be defined as in Theorem \ref{thm:
connected}. Then $A^{\alpha}\cong B$.
\end{proposition}

\begin{proof} It suffices to check that the conditions in Proposition
\ref{pro:gamma} hold.

Let $T=\left\{  \tau_{1},\ldots,\tau_{s}\right\}  .$ For
$i\in\left\{ 1,\ldots,\theta\right\}  $  there exists a unique
$\tau\left( i\right) \in T$ such that $i\in C_{\tau\left(  i\right)
}$. Since $\alpha_{\tau_j} = \varepsilon_{A \otimes A}$ on the
sub-bialgebra of $A$ generated by the $x_i$, $i \notin C_{\tau_j}$,
then for any $z \in A$, $0 \leq u \leq r_i -1$,
\begin{equation} \label{form:alphazi}
\alpha\left(  z\otimes x_{i}^{u}\right) =\alpha_{\tau_{1}}
\ast\cdots\ast\alpha_{\tau_{s}}\left(  z\otimes x_{i}^{ u}\right) =
\alpha_{\tau(i)} \left( z\otimes x_{i}^{ u}\right), \text{ and }
\alpha\left(   x_{i}^{u} \otimes z\right)   = \alpha_{\tau(i)}
\left( x_{i}^{ u}\otimes z \right).
\end{equation}
Now (\ref{form:gamma1})- (\ref{form:gamma4}) follow immediately from
the definitions of the cocycles $\gamma_i$ and $\gamma_{ij}$.
\end{proof}
\vspace{2mm}
\begin{center}
\textbf{Acknowledgement}:  Thanks to the referee for his/her careful
reading of this note.
\end{center}

\end{document}